\documentclass[12pt,fleqn,leqno]{amsart}
\usepackage{amssymb}
\usepackage{mathrsfs}
\usepackage{enumitem}
\usepackage[bb=boondox]{mathalfa} 
\usepackage[svgnames]{xcolor}
\usepackage[colorlinks,linkcolor=blue,citecolor=Green]{hyperref}
\usepackage{titlesec}
\usepackage{tikz}

\titleformat{\section}[block]
 {\bfseries}
 {\thesection.}
 {\fontdimen2\font}
 {}
 
\newtheorem{theorem}{Theorem}[section]
\newtheorem{lemma}[theorem]{Lemma}
\newtheorem{corollary}[theorem]{Corollary}
\newtheorem{proposition}[theorem]{Proposition}

\theoremstyle{definition}
\newtheorem{remark}[theorem]{Remark}
\newtheorem{question}{Question}
\newtheorem{example}[theorem]{Example}

\DeclareMathOperator{\N}{\mathbb{N}}

\DeclareMathOperator{\uhr}{\upharpoonright} 
\DeclareMathOperator{\sel}{\mathnormal{se}\mkern-1.5mu\ell}
\DeclareMathOperator{\cl}{cl} 

\setlist{noitemsep}
\setenumerate{labelindent=\parindent,label=\upshape{(\alph*)}}

\renewcommand{\emptyset}{\varnothing}
\numberwithin{equation}{section}
\begin{document}

\author{Valentin Gutev}

\address{Department of Mathematics, Faculty of Science, University of
   Malta, Msida MSD 2080, Malta}
\email{valentin.gutev@um.edu.mt}



\title{Coarser Compact Topologies}

\begin{abstract}
  It is introduced the concept of a quasi-king space, which is a
  natural generalisation of a king space. In the realm of suborderable
  spaces, king spaces are precisely the compact spaces, so are the
  quasi-king spaces. In contrast, quasi-king spaces are more flexible
  in handling coarser selection topologies. The main purpose of this
  paper is to show that a weakly orderable space is quasi-king if and
  only if all of its coarser selection topologies are compact.
\end{abstract}

\date{\today}
\maketitle

\section{Introduction}

All spaces are Hausdorff topological spaces. For a set $X$, let
\[
\mathscr{F}_2(X)=\{S\subset X: 1\leq |S| \leq 2\}.
\]
A map $\sigma:\mathscr{F}_2(X)\to X$ is a \emph{weak selection} for
$X$ if $\sigma(S)\in S$ for every $S\in \mathscr{F}_2(X)$.  Every weak
selection $\sigma$ for $X$ generates an order-like relation
$\leq_\sigma$ on $X$ defined by $x\leq_\sigma y$ if
$\sigma(\{x,y\})=x$ \cite[Definition 7.1]{michael:51}; and we write
$x<_\sigma y$ if $x\leq_\sigma y$ and $x\neq y$.  The relation
$\leq_\sigma$ is similar to a linear order being both total and
antisymmetric, but is not necessarily transitive.  If $X$ is a
topological space, then $\sigma$ is \emph{continuous} if it is
continuous with respect to the Vietoris topology on
$\mathscr{F}_2(X)$. This can be expressed only in terms of
$\leq_\sigma$ by the property that for every $x,y\in X$ with
$x<_\sigma y$, there are open sets $U,V\subset X$ such that $x\in U$,
$y\in V$ and $s<_\sigma t$ for every $s\in U$ and $t\in V$, see
\cite[Theorem 3.1]{gutev-nogura:01a}. \medskip

In 1921, studying dominance hierarchy in chickens and other birds,
Thorleif Schjelderup-Ebbe coined the term ``pecking order''.
Subsequently, in 1951, H. G. Landau \cite{MR0041412} (see also
\cite{MR567954}) used this `order' to show that any finite flock of
chickens has a most dominant one, called a \emph{king}. Landau's
mathematical model was based on Graph Theory and became known as ``The
King Chicken Theorem''. The pecking order is rarely linear, in fact it
is equivalent to the existence of a weak selection $\sigma$ on the
flock $X$. In this interpretation, an element $q\in X$ is called a
\emph{$\sigma$-king} if for every $x\in X$ there exists $y\in X$ with
$x\leq_\sigma y\leq_\sigma q$. Thus, Landau actually showed that each
weak selection $\sigma$ on a finite set $X$ has a
$\sigma$-king. Extending Landau's result to the setting of topological
spaces, Nagao and Shakhmatov called a space $X$ to be a \emph{king
  space} \cite{MR2944781} if $X$ has a continuous weak selection, and
every continuous weak selection $\sigma$ for $X$ has a
$\sigma$-king. Next, they showed that every compact space with a
continuous weak selection is a king space \cite[Theorem
2.3]{MR2944781}. In the inverse direction, Nagao and Shakhmatov showed
that each linearly ordered king space is compact (\cite[Corollary
3.3]{MR2944781}); also that each king space which is either
pseudocompact, or zero-dimensional, or locally connected, is compact
as well (\cite[Theorem 3.5]{MR2944781}). Subsequently, answering a
question of \cite{MR2944781}, it was shown in \cite[Theorem
4.1]{MR3640030} that each locally pseudocompact king space is also
compact.\medskip

On the other hand, there are simple examples of connected king spaces
which are not compact. For instance, such a space is the topological
sine curve
\[
  X=\{(0,0)\}\cup \{(t,\sin 1/t): 0<t\leq 1\}.
\]
However, $X$ has a coarser topology --- that of the interval $[0,1]$,
which is compact and admits the same compatible pecking orders (i.e.\
the same continuous weak selections). In this paper, we address such
spaces, and study the compactness of coarser topologies induced by
weak selections. To state our main result, we briefly recall some
related terminology.  Any weak selection $\sigma$ for $X$ generates a
natural topology $\mathscr{T}_{\sigma}$ on $X$
\cite{gutev-nogura:01a}, called a \emph{selection topology} and
defined following the pattern of the open interval topology, see
Section \ref{sec:selection-topologies}. If $X$ is a space and $\sigma$
is continuous, then $\mathscr{T}_\sigma$ is a coarser topology on $X$,
but $\sigma$ is not necessarily continuous with respect to
$\mathscr{T}_\sigma$ \cite{gutev-nogura:01a} (see also
\cite{gutev-nogura:03b,gutev-tomita:07}). A weak selection $\sigma$
for a space $X$ is called \emph{properly continuous} if
$\mathscr{T}_\sigma$ is a coarser topology on $X$ and $\sigma$ is
continuous with respect to $\mathscr{T}_\sigma$ \cite[Definition
4.4]{gutev-nogura:09a}. Thus, every properly continuous weak selection
is continuous, but the converse is not necessarily true. For a weak
selection $\sigma$ for $X$, we will say that a point $q\in X$ is a
\emph{quasi $\sigma$-king} if for each $x\in X$ there are finitely
many points $y_1,\dots,y_n\in X$ with
$x\leq_\sigma y_1\leq_\sigma\dots \leq_\sigma y_n\leq q$. Finally, we
shall say that $X$ is a \emph{quasi-king space} if $X$ has a weak
selection $\sigma$ such that $\mathscr{T}_\sigma$ is a coarser
topology on $X$, and each such weak selection $\sigma$ has a quasi
$\sigma$-king. The following theorem will be proved in this paper.

\begin{theorem}
  \label{theorem-KS-Coarser-v11:1}
  Let $X$ be a quasi-king space with a properly continuous weak
  selection. Then each coarser selection topology on $X$ is compact. 
\end{theorem}

Regarding the proper place of Theorem \ref{theorem-KS-Coarser-v11:1},
let us remark that a quasi-king space is a relaxed version of a
king space allowing dominance in several intermediate steps. Using
this, in Section \ref{sec:quasi-king-spaces} we give a simple direct
proof that each weak selection $\sigma$ on a set $X$, which generates a
compact selection topology $\mathscr{T}_\sigma$, has a quasi
$\sigma$-king (Theorem \ref{theorem-KS-Coarser-v4:1}). In the same
section, we also give an example of a space $X$ with a continuous weak
selection $\sigma$ which admits a quasi $\sigma$-king, but has no
$\sigma$-king (Example \ref{example-KS-Coarser-v9:2}). On the other
hand, all mentioned results for king spaces remain valid for
quasi-king spaces. Namely, in Section \ref{sec:clop-comp-spac} we show
that each suborderable quasi-king space is compact (Proposition
\ref{proposition-KS-Coarser-v4:2}), which is an element in the proof
of Theorem \ref{theorem-KS-Coarser-v11:1}.  This brings the following
natural question.

\begin{question}
  \label{question-KS-Coarser-v6:1}
  Does there exist a quasi-king space which is not a king space?
\end{question}

Theorem \ref{theorem-KS-Coarser-v11:1} also gives a partial solution
to a problem in the theory of continuous weak selections. Briefly, a
space is called \emph{weakly orderable} if it has a coarser orderable
topology, see Section~\ref{sec:selection-topologies}. Ernest Michael
showed that each connected space with a continuous weak selection is
weakly orderable \cite[Lemma 7.2]{michael:51}. Subsequently, Jan van
Mill and Evert Wattel showed that in the realm of compact spaces,
Michael's result remains valid without connectedness, namely that each
compact space with a continuous weak selection is (weakly) orderable
\cite[Theorem 1.1]{mill-wattel:81}. This led them to pose the question
whether a space with a continuous weak selection is weakly orderable;
the question itself became known as the \emph{weak orderability
  problem}. In 2009, Michael Hru{\v s}{\'a}k and Iv{\'a}n
Mart{\'\i}nez-Ruiz gave a counterexample by constructing a separable,
first countable and locally compact space which admits a continuous
weak selection but is not weakly orderable \cite[Theorem
2.7]{hrusak-martinez:09}; the interested reader is also referred to
\cite{gutev-nogura:09a} where the construction was discussed in
detail. However, this counterexample is a special Isbell-Mr\'{o}wka
space which is not normal. Thus, the weak orderability problem still
remains open in the realm of normal spaces, see \cite[Question
5]{gutev-nogura:09a}.  Another special case of this problem was
proposed in \cite{gutev-nogura:09a}, it is based on the fact that each
weakly orderable space has a properly continuous weak selection
\cite[Corollary 4.5]{gutev-nogura:09a}. Namely, the following question
was raised in \cite[Question 3]{gutev-nogura:09a}, also in
\cite[Problem 4.31]{gutev-2013springer}.

\begin{question}
  \label{question-KS-Coarser-v11:2}
  Let $X$ be a space which has a properly continuous weak
  selection. Then, is it true that $X$ is weakly orderable?
\end{question}

An essential element in the proof of Theorem
\ref{theorem-KS-Coarser-v11:1} is that in the realm of quasi-king
spaces, the answer to Question \ref{question-KS-Coarser-v11:2} is in
the affirmative, see Corollary
\ref{corollary-KS-Coarser-v11:1}.\medskip

The paper is organised as follows. In the next section, we give a
brief account on various orderable-like spaces. The idea of quasi-king
spaces is discussed in Section \ref{sec:quasi-king-spaces}. In Section
\ref{sec:clop-comp-spac}, we show a special case of Theorem
\ref{theorem-KS-Coarser-v11:1} that each clopen cover of a weakly
orderable quasi-king space has a finite subcover (Theorem
\ref{theorem-KS-Coarser-v1:1}). This is used further in Section
\ref{sec:coars-comp-select} to show that each coarser selection
topology on a weakly orderable quasi-king space is compact (Theorem
\ref{theorem-KS-Coarser-v3:1}). The proof of Theorem
\ref{theorem-KS-Coarser-v11:1} is finally accomplished in Section
\ref{sec:coars-pseud-topol} by showing that for a quasi-king space,
the selection topology induced by any properly continuous weak
selection is pseudocompact, hence compact as well (Theorem
\ref{theorem-KS-Coarser-v10:1}).

\section{Selection Topologies}
\label{sec:selection-topologies}

Let $\sigma$ be a weak selection for $X$, and $\leq_\sigma$ be the
order-like relation generated by $\sigma$, see the Introduction. For
subsets $A,B\subset X$, we write that $A\leq_\sigma B$ ($A<_\sigma B$)
if $x\leq_\sigma y$ (respectively, $x<_\sigma y$) for every $x\in A$
and $y\in B$. For a singleton $A=\{x\}$, we will simply write
$x\leq_\sigma B$ or $x<_\sigma B$ instead of $\{x\}\leq_\sigma B$ or
$\{x\}<_\sigma B$; in the same way, we write $A\leq_\sigma y$ or
$A<_\sigma y$ for a singleton $B=\{y\}$. Finally, we will use the
standard notation for the intervals generated by $\leq_\sigma$. For
instance, $(\leftarrow, x)_{\leq_\sigma}$ will stand for all $y\in X$
with $y<_\sigma x$; $(\leftarrow, x]_{\leq_\sigma}$ for that of all
$y\in X$ with $y\leq_\sigma x$; the $\leq_\sigma$-intervals
$(x,\to)_{\leq_\sigma}$ and $[x,\to)_{\leq_\sigma}$ are similarly
defined. However, working with such intervals should be done with
caution keeping in mind that the relation $\leq_\sigma$ is not
necessarily transitive. \medskip

Each weak selection $\sigma$ for $X$ generates a natural topology
$\mathscr{T}_{\sigma}$ on $X$, called a \emph{selection topology}
\cite{gutev-nogura:01a, gutev-nogura:03b}. It is patterned after the
open interval topology by taking the collection
$\big\{(\gets,x)_{\leq_\sigma}, (x,\to)_{\leq_\sigma}: x\in X\big\}$
as a subbase for $\mathscr{T}_\sigma$. Thus,
$\mathscr{T}_{\sigma}=\mathscr{T}_{\leq_\sigma}$ is the usual open
interval topology, whenever $\leq_\sigma$ is a linear order on
$X$. Each selection topology $\mathscr{T}_\sigma$ is Tychonoff
\cite[Theorem 2.7]{hrusak-martinez:09b}.  On the other hand,
$\mathscr{T}_\sigma$ may lack several of the other strong properties
of the open interval topology,
see~\cite{garcia-tomita:08,gutev-tomita:07}.\medskip

If $\sigma$ is a continuous weak selection for a topological space
$(X,\mathscr{T})$, then $\mathscr{T}_\sigma\subset \mathscr{T}$. The
converse is not true, and the inclusion $\mathscr{T}_\sigma\subset
\mathscr{T}$ does not imply continuity of $\sigma$ even in the realm
of compact spaces, see \cite[Example
1.21]{artico-marconi-pelant-rotter-tkachenko:02}, \cite[Example
3.6]{gutev-nogura:01a} and \cite[Example 4.3]{gutev-nogura:09a}. In
particular, a continuous weak selection $\sigma$ is not necessarily
continuous with respect to $\mathscr{T}_\sigma$. Based on this, a weak
selection $\sigma$ for a space $(X,\mathscr{T})$ was called

\begin{enumerate}[label=\upshape{(\roman*)}]
\item \emph{separately continuous} if $\mathscr{T}_\sigma\subset
  \mathscr{T}$ \cite{artico-marconi-pelant-rotter-tkachenko:02,
    gutev-nogura:09a}; and
\item \emph{properly continuous} if $\mathscr{T}_\sigma\subset
  \mathscr{T}$ and $\sigma$ is continuous with respect to
  $\mathscr{T}_\sigma$~\cite{gutev-nogura:09a}.
\end{enumerate}
Thus, each properly continuous weak selection is continuous, and each
continuous one is separately continuous, but none of these
implications is reversible. \medskip

In what follows, for a weak selection $\sigma$ for $X$, we will write
$\sigma\uhr Z$ to denote the restriction of $\sigma$ on a subset
$Z\subset X$, i.e.\ $\sigma\uhr Z=\sigma\uhr
\mathscr{F}_2(Z)$. Similarly, for a topology $\mathscr{T}$ on $X$, we
will use $\mathscr{T}\uhr Z$ for the subspace topology on $Z$.  The
following properties are evident from the definitions, and are left to
the reader.

\begin{proposition}
  \label{proposition-KS-Coarser-v11:1}
  Let $\sigma$ be a weak selection for $X$. Then
  \begin{enumerate}[label=\upshape{(\roman*)}]
  \item $\mathscr{T}_{\sigma\uhr Z}\subset \mathscr{T}_\sigma\uhr Z$,
    whenever $Z\subset X$\textup{;}
  \item $\sigma$ is separately continuous with respect to
    $\mathscr{T}_\sigma$\textup{;}
  \item $\sigma$ is continuous with respect to $\mathscr{T}_\sigma$,
    whenever $\leq_\sigma$ is a linear order on $X$.
  \end{enumerate}
\end{proposition}

For a topology $\mathscr{T}$ on $X$, we will use the prefix
``$\mathscr{T}$-'' to express properties of subsets of $X$ with
respect to this topology.  If $\sigma$ is a continuous selection for a
connected space $X$, then $\leq_\sigma$ is a linear order on $X$ and
$\mathscr{T}_\sigma$ is a coarser topology on $X$ \cite[Lemma
7.2]{michael:51}, which gives that $X$ is weakly orderable with
respect to $\leq_\sigma$. The property remains valid for separately
continuous weak selections, and will play an important role in the
paper.

\begin{proposition}
  \label{proposition-KS-Coarser-v9:1}
  Let $\sigma$ be a weak selection for $X$ and $Z\subset X$ be a
  $\mathscr{T}_\sigma$-connected subset of $X$. Then
  \begin{enumerate}[label=\upshape{(\roman*)}]
  \item\label{item:4} $x\notin Z$\ \ if and only if\ \ $x<_\sigma Z$
    or $Z<_\sigma x$\textup{;}
  \item\label{item:6} $\mathscr{T}_{\sigma\uhr Z}=\mathscr{T}_\sigma\uhr Z$ is
    the subspace topology on $Z$\textup{;}
  \item\label{item:5} $\leq_\sigma$ is a linear order on $Z$.
  \end{enumerate}
  In particular, $\sigma\uhr Z$ is a continuous weak selection for
  $(Z,\mathscr{T}_{\sigma\uhr Z})$.
\end{proposition}

\begin{proof}
  The property in \ref{item:4} is \cite[Proposition 2.4]{MR3430989},
  while \ref{item:6} is \cite[Proposition 2.5]{MR3430989}. The
  property \ref{item:5} is \cite[Proposition
  2.2]{gutev-nogura:01a}. The second part now follows from Proposition
  \ref{proposition-KS-Coarser-v11:1}, see also \cite[Proposition
  1.22]{artico-marconi-pelant-rotter-tkachenko:02}.
\end{proof}

Let $\mathscr{D}$ be a partition of $X$ and $\gamma$ be a weak
selection for $\mathscr{D}$. Following the idea of lexicographical
sums of linear orders, to each collection of weak selections
$\eta_\Delta:\mathscr{F}_2(\Delta)\to \Delta$, for
$\Delta\in \mathscr{D}$, we will associate the weak selection $\sigma$
for $X$ defined by
\begin{equation}
  \label{eq:KS-Coarser-v12:2}
  \begin{cases}
    \sigma\uhr \Delta=\eta_\Delta, &\text{for every $\Delta\in
  \mathscr{D}$,}\\
  \Gamma<_\sigma \Delta, &\text{whenever $\Gamma,\Delta\in
    \mathscr{D}$ with $\Gamma<_\gamma \Delta$.}
\end{cases}
\end{equation}
We will refer to $\sigma$ as the \emph{lexicographical $\gamma$-sum}
of $\eta_\Delta$, $\Delta\in \mathscr{D}$, or simply as the
\emph{lexicographical sum}, and will denote it by
$\sigma=\sum_{(\gamma,\Delta\in \mathscr{D})}\eta_\Delta$. In case
$\eta_\Delta=\eta\uhr \Delta$, $\Delta\in \mathscr{D}$, for some weak
selection $\eta$ for $X$, the lexicographical sum
$\sum_{(\gamma,\Delta\in \mathscr{D})} \eta_\Delta$ was used in
\cite{MR3430989} under the name of a
\emph{$(\mathscr{D},\gamma)$-clone} of $\eta$.

\begin{proposition}
  \label{proposition-KS-Coarser-v12:2}
  Let $\mathscr{D}$ be an open partition of a space $X$, $\gamma$ be a
  weak selection for $\mathscr{D}$, and $\eta_\Delta$ be a
  separately continuous weak selection for $\Delta$, for each
  ${\Delta\in \mathscr{D}}$. Then the lexicographical $\gamma$-sum
  $\sigma=\sum_{(\gamma,\Delta\in \mathscr{D})}\eta_\Delta$ is a
  separately continuous weak selection for $X$. Moreover, $\sigma$ is
  continuous provided so is each $\eta_\Delta$, $\Delta\in
  \mathscr{D}$.
\end{proposition}

\begin{proof}
  Let $\Delta\in \mathscr{D}$ and $x\in \Delta$. According to
  (\ref{eq:KS-Coarser-v12:2}), we have that 
  \[
  (\gets,x)_{\leq_\sigma}= (\gets,x)_{\leq_{\eta_{\Delta}}}\cup
  \bigcup_{\Gamma<_\gamma \Delta} \Gamma.
  \]
  Hence, $(\gets,x)_{\leq_\sigma}$ is open in $X$ because
  $\eta_{\Delta}$ is separately continuous and $\mathscr{D}$ consists
  of open sets.  Similarly, $(x,\to)_{\leq_\sigma}$ is also
  open. Thus, $\sigma$ is separately continuous.\smallskip

  Suppose that each $\eta_\Delta$, $\Delta\in \mathscr{D}$, is
  continuous. To show that $\sigma$ is also continuous, take
  $p,q\in X$ with $p<_\sigma q$. It now suffices to find open sets
  $U,V\subset X$ such that $p\in U$, $q\in V$ and $U<_\sigma V$. To
  this end, let $\Delta_p,\Delta_q\in \mathscr{D}$ be the unique
  elements with $p\in \Delta_p$ and $q\in \Delta_q$. If
  $\Delta_p\neq \Delta_q$, then by (\ref{eq:KS-Coarser-v12:2}),
  $\Delta_p<_\sigma \Delta_q$ and we can take $U=\Delta_p$ and
  $V=\Delta_q$ because $\mathscr{D}$ consists of open sets. If
  $\Delta_p=\Delta_q=\Delta$, we can use that
  $\sigma\uhr \Delta= \eta_\Delta$ is continuous to take open sets
  $U,V\subset \Delta$ such that $p\in U$, $q\in V$ and
  $U<_{\eta_\Delta} V$. Evidently, $U<_\sigma V$.
\end{proof}

\section{Quasi-King Spaces}
\label{sec:quasi-king-spaces}

Let $\sigma$ be a weak selection for $X$, and
$\ll_\sigma,\lll_\sigma\subset X^2$ be the binary relations defined
for $x,y\in X$ by
\begin{equation}
  \label{eq:KS-Coarser-v4:1}
  \left\{\begin{aligned}
  x\ll_\sigma y\quad&\text{if}\ x\leq_\sigma y_1\leq_\sigma y,\ \
  \text{for some $y_1\in X$,\quad and}\\
  x\lll_\sigma y\quad &\text{if}\ x\leq_\sigma
  y_1\leq_\sigma \dots\leq_\sigma y_n\leq_\sigma y,\ \
  \text{for some $y_1,\dots, y_n\in X$.} 
\end{aligned}\right.
\end{equation}
It is evident that $\leq_\sigma\subset\ll_\sigma\subset\lll_\sigma$,
and that $\ll_\sigma$ and $\lll_\sigma$ are total and reflexive
because so is $\leq_\sigma$.  Furthermore, $\lll_\sigma$ is always
transitive. However, in general, $\ll_\sigma$ and $\lll_\sigma$ are
not antisymmetric, and  may contain properly the relation
$\leq_\sigma$. In fact, $\leq_\sigma$ is equal to one of these
relations precisely when $\leq_\sigma$ is transitive (i.e.\ a linear
order), which is summarised in the proposition below.

\begin{proposition}
  \label{proposition-KS-Coarser-v9:2}
  Let $\sigma$ be a weak selection for $X$. Then
  $\lll_\sigma=\leq_\sigma$ if and only if $\ll_\sigma=\leq_\sigma$,
  which is in turn equivalent to $\leq_\sigma$ being transitive.
\end{proposition}

\begin{proof}
  Evidently, $\lll_\sigma= \leq_\sigma$ implies that
  $\ll_\sigma=\leq_\sigma$ because $\leq_\sigma\subset
  \ll_\sigma\subset \lll_\sigma$.  If $\leq_\sigma$ is not transitive,
  then $X$ contains points $x,y,z\in X$ with $x<_\sigma y<_\sigma
  z<_\sigma x$. In this case, $\ll_\sigma\neq \leq_\sigma$
  because $x<_\sigma y\ll_\sigma x$.
\end{proof}

Our interest in these binary relations is the interpretation that
$p\in X$ is a \emph{$\sigma$-king} if $x\ll_\sigma p$ for all
$x\in X$; and $p$ is a \emph{quasi $\sigma$-king} if $x\lll_\sigma p$
for all $x\in X$, see the Introduction. In other words, the
$\sigma$-kings of $X$ are the $\ll_\sigma$-maximal elements of $X$,
and the quasi $\sigma$-kings are the $\lll_\sigma$-maximal ones. We
proceed with some examples about the difference between $\sigma$-kings
and quasi $\sigma$-ones.

\begin{example}
  \label{example-KS-Coarser-v9:1}
  Let $X=\{a,b,c,p\}$ consist of four points, and $\gamma$ be the weak
  selection for $X$ defined by $a<_\gamma b<_\gamma c <_\gamma a$ and
  $c<_\gamma p<_\gamma \{a,b\}$. Graphically, $\leq_\gamma$ is
  represented by the diagram below, where
  ``$<_\gamma$''=``$\leftarrow$'' and the shortest chain
  $a\leftarrow \dots \leftarrow p$ of arrows illustrating the
  relation  $a\lll_\gamma p$ is emphasised.
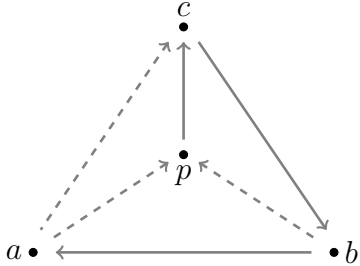
\begin{center}
  \begin{tikzpicture}
  \filldraw (2,-1) circle [radius=1.5pt] node[right] {$b$} (-2,-1) circle
  [radius=1.5pt] node[left] {$a$}
  (0,2) circle [radius=1.5pt] node[above] {$c$};   
  \draw[gray,line width=1pt, <-] (-1.7,-1) -- (1.7,-1);
  \draw[gray,line width=1pt,<-] (1.9,-.7) -- (.2,1.8); 
  \draw[gray, line width=1pt,<-,dashed] (-.2,1.8) -- (-1.9,-.7); 
  \filldraw (0,.3) circle [radius=1.5pt] node[anchor=north] {$p$};
  \draw[gray,line width=1pt,<-] (0,1.8) -- (0,.5); 
  \draw[gray,line width=1pt,<-,dashed] (-.2,.2) -- (-1.8,-.85);
  \draw[gray, line width=1pt,<-,dashed] (.2,.2) -- (1.8,-.85); 
  \end{tikzpicture}
\end{center}    
Then $p$ is a quasi $\gamma$-king for $X$, but not a $\gamma$-king. On
the other hand, $a$, $b$ and $c$ are $\gamma$-kings for
$X$.\hfill\textsquare
\end{example}

In case of infinite spaces, we have the following similar example
where all points of $X$ are quasi $\sigma$-kings for some continuous
weak selection $\sigma$, but $X$ has no $\sigma$-king. 

\begin{example}
  \label{example-KS-Coarser-v9:2}
  Following Example \ref{example-KS-Coarser-v9:1}, let
  $X=\Delta_a\uplus \Delta_b\uplus \Delta_c$ be the topological sum of
  three copies $\Delta_a$, $\Delta_b$ and $\Delta_c$ of the interval
  $(0,1)$, and let $\gamma$ be the weak selection on open partition
  $\mathscr{D}=\big\{\Delta_{a},\Delta_{b},\Delta_{c}\big\}$ defined
  by
  $ \Delta_{a}<_\gamma \Delta_{b}<_\gamma \Delta_{c}<_\gamma
  \Delta_{a}$. Take the standard selection
  $\eta(\{x,y\})=\min\{x,y\}$, $x,y\in (0,1)$, on each one of the open
  segments $\Delta_a$, $\Delta_b$ and $\Delta_c$. Finally, let
  $\sigma$ be the lexicographical $\gamma$-sum of these selections. In
  other words, $\sigma$ is the weak selection for $X$ which is
  continuous on each of these open segments, and
  $\Delta_{a}<_\sigma \Delta_{b}<_\sigma \Delta_{c}<_\sigma
  \Delta_{a}$.  According to Proposition
  \ref{proposition-KS-Coarser-v12:2}, $\sigma$ is
  continuous. Moreover, each element of $X$ is a quasi $\sigma$-king,
  but $X$ has no $\sigma$-king because none of the open segments
  contains a last element with respect to
  $\leq_\sigma$.\hfill\textsquare
\end{example}

Regarding the existence of quasi $\sigma$-kings, we have the
following natural result  which is complementary to \cite[Theorem
2.3]{MR2944781}.

\begin{theorem}
\label{theorem-KS-Coarser-v4:1}
Let $\sigma$ be a weak selection for $X$ such that
$\mathscr{T}_\sigma$ is a compact topology on $X$. Then $X$ has a
quasi $\sigma$-king.
\end{theorem}

\begin{proof}
  For every $x\in X$, let
  \begin{equation}
    \label{eq:KS-Coarser-v9:3}
    K_{x}=\{p\in X: x\lll_\sigma p\}.
  \end{equation}
  Evidently, each $K_x$ is nonempty because $x\in K_x$.  Take
  $x,y\in X$ with $x\leq_\sigma y$, and $p\in K_y$. Then $p\in K_x$
  because $x\leq_\sigma y\lll_\sigma p$ implies $x\lll_\sigma p$, see
  \eqref{eq:KS-Coarser-v4:1}. Thus, every two elements of the
  collection $\{K_x:x\in X\}$ are comparable by inclusion. Hence,
  it has the finite intersection property.  Let
  $\cl_{\mathscr{T}_\sigma}(A)=\overline{A}^{\,\mathscr{T}_\sigma}$ be
  the closure of a subset $A\subset X$ in the topology
  $\mathscr{T}_\sigma$. Since $\mathscr{T}_\sigma$ is a compact
  topology, we get that
  $\bigcap_{x\in X}\cl_{\mathscr{T}_\sigma}(K_x)\neq\emptyset$. Let
  $p\in \bigcap_{x\in X}\cl_{\mathscr{T}_\sigma}(K_x)$. If
  $x\leq_\sigma p$ for every $x\in X$, then clearly $p$ is a
  $\sigma$-king for $X$. If $p<_\sigma q$ for some $q\in X$, then $q$
  is a quasi $\sigma$-king for $X$. Indeed, for every $x\in X$ there
  exists $p_x\in K_x$ with $p_x<_\sigma q$, because
  $p\in (\leftarrow,q)_{\leq_\sigma}\cap
  \cl_{\mathscr{T}_\sigma}(K_x)$. According to
  (\ref{eq:KS-Coarser-v4:1}) and (\ref{eq:KS-Coarser-v9:3}),
  $q\in K_x$ for every $x\in X$.
\end{proof}

Recall that a space $X$ is \emph{quasi-king} if it has a separately
continuous weak selection, and each separately continuous weak
selection $\sigma$ for $X$ has a quasi $\sigma$-king.  We now have the
following consequence, compare with \cite[Theorem 2.3]{MR2944781}.

\begin{corollary}
  \label{corollary-KS-Coarser-v4:2}
  Let $X$ be a space with a separately continuous weak selection. If
  each coarser selection topology on $X$ is compact, then $X$ is a
  quasi-king space.
\end{corollary}

We conclude with some remarks. 

\begin{remark}
  \label{remark-KS-Coarser-v12:3}
  The proof of Theorem \ref{theorem-KS-Coarser-v4:1} does not follow
  from that of \cite[Theorem 2.3]{MR2944781}. In fact, the author is
  unaware if, in the setting of Theorem \ref{theorem-KS-Coarser-v4:1},
  $X$ has a $\sigma$-king.
\end{remark}

\begin{remark}
  \label{remark-KS-Coarser-v12:2}
  Following the idea of Example \ref{example-KS-Coarser-v9:2}, one can
  easily characterise the spaces in which each quasi $\sigma$-king is
  a $\sigma$-king. Namely, for a quasi-king space $X$, the
  following are equivalent:
  \begin{enumerate}
  \item\label{item:KS-Coarser-vgg-rev:1} $\lll_\sigma=\ll_\sigma$, for
    each separately continuous weak selection $\sigma$ for $X$.
  \item\label{item:KS-Coarser-vgg-rev:2} $X$ is the topological sum of
    at most three connected subsets.
  \end{enumerate}

  Here, the requirement that $X$ is a quasi-king space is
  important. Indeed, the space in Example
  \ref{example-KS-Coarser-v9:2} satisfies
  \ref{item:KS-Coarser-vgg-rev:2}, but is not quasi-king. So,
  implicitly, such a partition of a quasi-king space $X$ must be of
  $\mathscr{T}_\sigma$-compact sets, for each (some) separately
  continuous weak selection $\sigma$ for $X$, see Propositions
  \ref{proposition-KS-Coarser-v9:1} and
  \ref{proposition-KS-Coarser-v4:2}.  Moreover,
  \ref{item:KS-Coarser-vgg-rev:2} implies that each separately
  continuous weak selection for $X$ is continuous (by Propositions
  \ref{proposition-KS-Coarser-v9:1} and
  \ref{proposition-KS-Coarser-v12:2}), therefore such quasi-king
  spaces are completely identical to kings spaces.
\end{remark}

\begin{remark}
  \label{remark-KS-Coarser-v16:1}
  Let $\sigma$ be a weak selection for $X$. Following \cite{MR567954},
  a point $p\in X$ will be called a \emph{$\sigma$-emperor} if it is
  the $\leq_\sigma$-maximal element of $X$, namely if $x\leq_\sigma p$
  for all $x\in X$. Thus, $X$ may have at most one $\sigma$-emperor,
  and each $\sigma$-emperor is a (quasi) $\sigma$-king. If $X$ is a
  finite set, then $X$ has exactly one $\sigma$-king if and only if
  that king is a $\sigma$-emperor \cite[Theorem 4]{MR567954}. In the
  setting of infinite sets, this is not necessarily true, and the
  property defines a special class of topological spaces. To this end,
  for convenience, let $\sel_2(X)$ be the collection of all weak
  selections for a set $X$. Then for a space $X$ with a separately
  continuous weak selection, the following are equivalent\textup{:}
  \begin{enumerate}
  \item\label{item:1} $X$ is $\mathscr{T}_\sigma$-compact and
    $\leq_\sigma$ is a linear order, for each separately continuous
    $\sigma\in\sel_2(X)$.
  \item Each separately continuous $\sigma\in\sel_2(X)$ has a
    $\sigma$-emperor.
  \item Each separately continuous $\sigma\in \sel_2(X)$ has exactly one
    quasi $\sigma$-king.
  \item Each separately continuous $\sigma\in\sel_2(X)$ has exactly one
    $\sigma$-king.
  \item $X$ is the topological sum of at most two
    $\mathscr{T}_\sigma$-compact sets, for each separately continuous
    $\sigma\in\sel_2(X)$.
  \end{enumerate}
  
  By Proposition \ref{proposition-KS-Coarser-v11:1}, the first
  condition implies that each separately continuous weak selection for
  $X$ is properly continuous.
\end{remark}

\section{Clopen Compactness}
\label{sec:clop-comp-spac}

Here, we show that every weakly orderable quasi-king space is compact
in the topology generated by its clopen subsets, which furnishes an
essential part in the proof of Theorem \ref{theorem-KS-Coarser-v11:1}.

\begin{theorem}
  \label{theorem-KS-Coarser-v1:1}
  Let $X$ be a weakly orderable quasi-king space. Then each clopen
  cover of $X$ has a finite subcover.
\end{theorem}

The proof of Theorem \ref{theorem-KS-Coarser-v1:1} is based on several
observations about quasi-king spaces. The next proposition shows that
the following property of king spaces is also valid for quasi-king
spaces, see \cite[Lemma 3.1]{MR2944781}.

\begin{proposition}
  \label{proposition-KS-Coarser-v4:1}
  If $X$ is a quasi-king space, then each clopen subset of $X$ is also
  a quasi-king space. 
\end{proposition}

\begin{proof}
  Let $A\subset X$ be a clopen set, and $\eta$ be a separately
  continuous weak selection for $A$. Since $X\setminus A$ is also
  clopen and has a separately continuous weak selection, it follows
  from Proposition \ref{proposition-KS-Coarser-v12:2} that $X$ has a
  separately continuous weak selection $\sigma$ with
  $\sigma\uhr A=\eta$ and $X\setminus A<_\sigma A$.  Then by
  hypothesis, $X$ has a quasi $\sigma$-king $p\in X$. For a point
  $x\in A$, this means that
  $x\leq_\sigma y_1\leq_\sigma \dots \leq_\sigma y_n\leq_\sigma p$,
  for some $y_1,\dots, y_n\in X$. However, $x\in A$ and
  $X\setminus A<_\sigma A$, which implies that
  $y_1,\dots, y_n,p\in A$. Accordingly, $p$ is a quasi $\eta$-king of
  $A$ because $\sigma\uhr A=\eta$.
\end{proof}

Subspaces of orderable spaces are not necessarily orderable, they are
called \emph{suborderable}. Their topology can be also described in
terms of ``order''-intervals. Briefly, a subset $\Delta\subset X$ of
an ordered set $(X,\leq)$ is called a \emph{$\leq$-interval}, or a
\emph{$\leq$-convex set}, if
$(a,b)_\leq=(a,\to)_\leq\cap(\gets,b)_\leq\subset \Delta$, for every
${a,b\in \Delta}$ with $a\leq b$. A topological space
$(X,\mathscr{T})$ is called \emph{generalised ordered} if it admits a
linear order $\leq$, called \emph{compatible}, such that the
corresponding open interval topology $\mathscr{T}_\leq$ is coarser
than the topology $\mathscr{T}$, and $\mathscr{T}$ has a base of
$\leq$-intervals. According to a result of E. \v{C}ech, generalised
ordered spaces are precisely the suborderable spaces, see e.g.\
\cite{cech:66,MR1795166}. We now get with ease that each suborderable
quasi-king space is compact, see \cite[Lemma 3.2 and Corollary
3.3]{MR2944781}.

\begin{proposition}
  \label{proposition-KS-Coarser-v4:2}
  Each suborderable quasi-king space is compact. 
\end{proposition}

\begin{proof}
  Let $X$ be a quasi-king space which is suborderable with respect to
  a linear order $\leq$. Then $\eta(\{x,y\})=\min_\leq \{x,y\}$,
  $x,y\in X$, is a continuous weak selection for $X$ with
  $\leq_\eta=\leq$. Hence, $X$ has a unique quasi $\eta$-king, which
  is the $\leq$-maximal element of $X$, see Proposition
  \ref{proposition-KS-Coarser-v9:2}. Since $X$ is also suborderable
  with respect to the reverse linear order, it has a $\leq$-minimal
  element as well. This implies that $X$ is actually orderable with
  respect to $\leq$. Indeed, let $E$ and $D$ be nonempty clopen
  subsets of $X$ such that $E< D$ and $X=E\cup D$.  By Proposition
  \ref{proposition-KS-Coarser-v4:1}, both $E$ and $D$ are quasi-king
  spaces. Hence, by what has been shown above, $E$ has a maximal
  element and $D$ has a minimal one. Thus, the pair $(E,D)$ is a
  clopen jump and, consequently, $X$ is orderable with respect to
  $\leq$, see e.g.\ \cite[Lemma 6.4]{gutev:07b}. This also implies
  that $X$ must be compact. Namely, each nonempty clopen set
  $A\subset X$ is both a quasi-king space (by Proposition
  \ref{proposition-KS-Coarser-v4:1}) and suborderable with respect to
  $\leq$. So, by the same token, it has maximal and minimal elements.
  Therefore, $X$ is compact \cite{haar-konig:10}, see also
  \cite[Proposition 6.1]{gutev:07b}.
\end{proof}

\begin{corollary}
  \label{corollary-KS-Coarser-v4:3}
  Let $X$ be a quasi-king space which is weakly orderable with respect
  to a linear order $\leq$.  Then the open interval topology
  $\mathscr{T}_{\leq}$ is a coarser compact topology on $X$. 
\end{corollary}

\begin{proof}
  The topology $\mathscr{T}_\leq$ is a coarser topology on $X$, and,
  in particular, each separately continuous weak selection for
  $(X,\mathscr{T}_\leq)$ is a separately continuous weak selection for
  $X$. Therefore, the orderable space $(X,\mathscr{T}_{\leq})$ is also
  quasi-king. Hence, by Proposition \ref{proposition-KS-Coarser-v4:2},
  $(X,\mathscr{T}_\leq)$ is compact.
\end{proof}

We are now ready for the proof of Theorem
\ref{theorem-KS-Coarser-v1:1}. 

\begin{proof}[Proof of Theorem \ref{theorem-KS-Coarser-v1:1}]
  Let $X$ be a weakly orderable space with respect to a linear order
  $\leq$. According to Corollary \ref{corollary-KS-Coarser-v4:3}, it
  suffices to show that each clopen subset of $X$ is open in
  $(X,\mathscr{T}_\leq)$. So, let $A\subset X$ be clopen in $X$. Then
  $A$ is quasi-king (by Proposition \ref{proposition-KS-Coarser-v4:1})
  and suborderable in the subspace topology $\mathscr{T}_\leq \uhr
  A$. In fact, $A$ is a quasi-king space with respect
  $\mathscr{T}_\leq \uhr A$ because $\mathscr{T}_\leq \uhr A$ is a
  coarser topology on $A$ and the weak selection $\min_{\leq}\{x,y\}$,
  $x,y\in A$, is continuous with respect to this topology (by
  Proposition \ref{proposition-KS-Coarser-v11:1}). Thus, by
  Proposition \ref{proposition-KS-Coarser-v4:2}, $A$ is a compact
  subset of $(X,\mathscr{T}_\leq)$. For the same reason, so is
  $X\setminus A$. Therefore, $A=X\setminus(X\setminus A)$ is open in
  $(X,\mathscr{T}_\leq)$.
\end{proof}

\section{Coarser Compact Selection Topologies}
\label{sec:coars-comp-select}

Here, we prove the following special case of
Theorem \ref{theorem-KS-Coarser-v11:1}.

\begin{theorem}
  \label{theorem-KS-Coarser-v3:1}
  Let $X$ be a weakly orderable quasi-king space, and $\sigma$ be a
  separately continuous weak selection for $X$. Then
  $\mathscr{T}_\sigma$ is a compact coarser topology on $X$.
\end{theorem}

The proof of Theorem \ref{theorem-KS-Coarser-v3:1} is based on 
properties of components relative to selection topologies. The
\emph{components} (called also \emph{connected components}) are the
maximal connected subsets of a space $X$. They form a closed partition
$\mathscr{C}[X]$ of $X$, and each point $x\in X$ is contained in a
unique component $\mathscr{C}[x]$ called the \emph{component} of $x$
in $X$. The \emph{quasi-component} $\mathscr{Q}[x]$ of a point
$x\in X$ is the intersection of all clopen subsets of $X$ containing
$x$. The quasi-components also form a partition $\mathscr{Q}[X]$ of
$X$, thus they are simply called \emph{quasi-components} of $X$. Each
component of a point is contained in the quasi-component of that
point, but the converse is not necessarily true. However, if $X$ has a
continuous weak selection, then $\mathscr{C}[x]= \mathscr{Q}[x]$ for
every $x\in X$ \cite[Theorem 4.1]{gutev-nogura:00b}. The property
remains valid for the components of selection topologies.

\begin{proposition}
  \label{proposition-KS-Coarser-v7:2}
  Let $\sigma$ be a weak selection for $X$. Then each quasi-component
  of $(X,\mathscr{T}_\sigma)$ is connected.
\end{proposition}

\begin{proof}
  By Proposition \ref{proposition-KS-Coarser-v11:1}, $\sigma$ is a
  separately continuous weak selection for
  $(X,\mathscr{T}_\sigma)$. Then the property follows from
  \cite[Corollary 2.3]{MR3430989}.
\end{proof}

Regarding Proposition \ref{proposition-KS-Coarser-v7:2}, let us
explicitly remark that if $C\subset X$ is a component of a space $X$
and $\sigma$ is a separately continuous weak selection for $X$, then
$C$ is also a connected subset of $(X,\mathscr{T}_\sigma)$. However,
$C$ is not necessarily a $\mathscr{T}_\sigma$-component, namely a
component of the space $(X,\mathscr{T}_\sigma)$. Keeping this in mind, we have
the following construction of clopen sets associated to
$\mathscr{T}_\sigma$-components.

\begin{proposition}
  \label{proposition-KS-Coarser-v10:1}
  Let $\eta$ be a weak selection for $X$, and $Z\subset X$ be a
  $\mathscr{T}_\eta$-component of $X$ which has no
  $\leq_\eta$-maximal element. Then $Z$ is contained in a
  $\mathscr{T}_\eta$-clopen set $Y\subset X$ with $Y\setminus
  Z<_\eta Z$.
\end{proposition}

\begin{proof}
  The set $Y=\bigcup_{z\in Z}(\leftarrow,z)_{\leq_\eta}$ is
  $\mathscr{T}_\eta$-open.  Moreover, $Z\subset Y$ because $Z$ has no
  last element with respect to $\leq_\eta$. If $y\in X\setminus Z$ and
  $y\leq_\eta z$ for some $z\in Z$, then $y<_\eta Z$ because $Z$ is
  $\mathscr{T}_\eta$-connected, see Proposition
  \ref{proposition-KS-Coarser-v9:1}.  This implies that
  $Y\setminus Z<_\eta Z$. It also implies that
  $Y=(\leftarrow,x]_{\leq_\eta}\cup Z$ for some (any) point $x\in
  Z$. Since both $(\leftarrow,x]_{\leq_\eta}$ and $Z$ are
  $\mathscr{T}_\eta$-closed, so is $Y$.
\end{proof}

We now have the following crucial property of selection topologies.

\begin{lemma}
  \label{lemma-KS-Coarser-v16:1}
  Let $\sigma$ be a weak selection for $X$ such that
  $(X,\mathscr{T}_\sigma)$ is a quasi-king space. Then each
  $\mathscr{T}_\sigma$-component is $\mathscr{T}_\sigma$-compact.
\end{lemma}

\begin{proof}
  Take a non-degenerate $\mathscr{T}_\sigma$-component $Z\subset
  X$. Then by Proposition \ref{proposition-KS-Coarser-v9:1},
  $(Z,\mathscr{T}_\sigma\uhr Z)$ is orderable with respect to
  $\leq_\sigma$ being a connected space. Hence, to show that $Z$ is
  $\mathscr{T}_\sigma$-compact, it now suffices to show that it has
  both $\leq_\sigma$-minimal and $\leq_\sigma$-maximal elements. To
  this end, we will use that $\sigma$ determines a unique
  `complementary' selection $\sigma^*:\mathscr{F}_2(X)\to X$, defined
  by $S=\big\{\sigma(S),\sigma^*(S)\big\}$, $S\in
  \mathscr{F}_2(X)$. The important property of $\sigma^*$ is that
  $\mathscr{T}_{\sigma^*}=\mathscr{T}_\sigma$ because
  $\leq_{\sigma^*}$ is reverse to $\leq_\sigma$. Thus, given a weak
  selection $\eta$ for $X$ with $\mathscr{T}_\eta=\mathscr{T}_\sigma$,
  it suffices to show that $Z$ has a $\leq_\eta$-maximal element. To
  see this, assume the contrary that $X$ has a weak selection $\eta$
  with $\mathscr{T}_\eta=\mathscr{T}_\sigma$, but $Z$ has no
  $\leq_\eta$-maximal element. Then by Proposition
  \ref{proposition-KS-Coarser-v10:1}, $Z$ is contained in a
  $\mathscr{T}_\eta$-clopen set $Y$ with $Y\setminus Z<_\eta Z$.
  Using that $\mathscr{T}_\eta=\mathscr{T}_\sigma$, it follows from
  Proposition \ref{proposition-KS-Coarser-v4:1} that
  $(Y,\mathscr{T}_\eta\uhr Y)$ is also a quasi-king space. Moreover,
  $\gamma=\eta\uhr Y$ is a separately continuous weak selection for
  $(Y,\mathscr{T}_\eta\uhr Y)$, hence $Y$ has a quasi $\gamma$-king
  $q\in Y$. Since $Y\setminus Z<_\gamma Z$ and $Z$ has no
  $\leq_\gamma$-maximal element, $q<_\gamma x$ for some $x\in Z$. For
  the same reason, $q<_\gamma y$, for every $y\in Y$ with
  $x\leq_\gamma y$, because $\leq_\gamma$ is a linear order on
  $Z$. Accordingly, $q$ cannot be a quasi $\gamma$-king for $Y$. A
  contradiction.
\end{proof}

Finally, we also have that each $\mathscr{T}_\sigma$-component has a
base of $\mathscr{T}_\sigma$-clopen sets.

\begin{proposition}
  \label{proposition-KS-Coarser-v5:2}
  Let $\sigma$ be a weak selection for $X$ such that
  $(X,\mathscr{T}_\sigma)$ is a quasi-king space. Then each
  $\mathscr{T}_\sigma$-component has a base of clopen sets in
  $(X,\mathscr{T}_\sigma)$.
\end{proposition}

  \begin{proof}
    A space is \emph{rim-finite} if it has a base of open sets whose
    boundaries are finite. Evidently, $(X,\mathscr{T}_\sigma)$ is
    rim-finite. Take a $\mathscr{T}_\sigma$-component $Z$ of $X$, and
    a $\mathscr{T}_\sigma$-open set $V\subset X$ with $Z\subset
    V$. Since $Z$ is $\mathscr{T}_\sigma$-compact (by Lemma
    \ref{lemma-KS-Coarser-v16:1}) and $(X,\mathscr{T}_\sigma)$ is
    rim-finite, there exists $W\in \mathscr{T}_\sigma$ such that
    $Z\subset W\subset V$ and the boundary of $W$ is finite. However,
    by Proposition \ref{proposition-KS-Coarser-v7:2}, $Z$ is also a
    quasi-component of $(X,\mathscr{T}_\sigma)$.  Hence, there exists
    a $\mathscr{T}_\sigma$-clopen set $U\subset X$ with
    $Z\subset U\subset W\subset V$.
\end{proof}

\begin{proof}[Proof of Theorem \ref{theorem-KS-Coarser-v3:1}]
  Let $\sigma$ be a separately continuous weak selection for $X$. Take
  an open cover $\mathscr{U}\subset \mathscr{T}_\sigma$ of $X$, and
  let $\mathscr{U}^F$ be the cover of $X$ consisting of all finite
  unions of elements of $\mathscr{U}$. According to Lemma
  \ref{lemma-KS-Coarser-v16:1} and Proposition
  \ref{proposition-KS-Coarser-v5:2}, $\mathscr{U}^F$ has a clopen
  refinement $\mathscr{V}$. Then by Theorem
  \ref{theorem-KS-Coarser-v1:1}, $\mathscr{V}$ has a finite
  subcover. This implies that $\mathscr{U}$ has a finite
  subcover as well.
\end{proof}

\section{Coarser Pseudocompact Topologies}
\label{sec:coars-pseud-topol}

For simplicity, we shall say that a space $(X,\mathscr{T})$ is
\emph{selection-orderable} if it has a continuous weak selection
$\varphi$ with $\mathscr{T}=\mathscr{T}_\varphi$. The main idea behind
this convention is that for a space $X$ with a properly continuous
weak selection $\varphi$, the space $(X,\mathscr{T}_\varphi)$ is
selection-orderable.\medskip

We now finalise the proof of Theorem \ref{theorem-KS-Coarser-v11:1} by
showing the following general result involving implicitly
pseudocompactness.

\begin{theorem}
  \label{theorem-KS-Coarser-v10:1}
  Each selection-orderable quasi-king space is compact.
\end{theorem}

To prepare for the proof of Theorem \ref{theorem-KS-Coarser-v10:1}, we
first extend the following property of king spaces to the case of
quasi-king spaces, see \cite[Lemma 3.4]{MR2944781}.

\begin{proposition}
  \label{proposition-KS-Coarser-v7:1}
  Let $X$ be a space with a continuous weak selection. If $X$ admits
  an infinite open partition, then $X$ is not a quasi-king space.
\end{proposition}

\begin{proof}
  Let $\mathscr{U}$ be an infinite open partition of $X$, and $\leq$
  be a linear order on $\mathscr{U}$ such that $\mathscr{U}$ has no
  last $\leq$-element. Take a weak selection $\gamma$ for
  $\mathscr{U}$ with $\leq_\gamma=\leq$. Also, for every
  $U\in \mathscr{U}$, take a continuous weak selection $\eta_U$ for
  $U$. Finally, let $\sigma$ be the lexicographical $\gamma$-sum of
  these selections. By Proposition \ref{proposition-KS-Coarser-v12:2},
  $\sigma$ is continuous. Moreover, $\sigma$ induces the same linear
  order on $\mathscr{U}$ as that of $\gamma$, see
  \eqref{eq:KS-Coarser-v12:2}. This implies that $\sigma$ has no quasi
  $\sigma$-king. Indeed, let $q\in V$ for some $V\in
  \mathscr{U}$. Next, using that $\mathscr{U}$ has no last
  $\leq_\sigma$-element, take any $U\in \mathscr{U}$ with
  $V<_\sigma U$. If $x\in U$ and $x\leq_\sigma y$, then $y$ has the
  same property as $x$ in the sense that $y\in W$ for some
  $W\in \mathscr{U}$ with $V<_\sigma W$. Hence, for any
  finite number of points $y_1, \dots , y_n\in X$ with
  $x\leq_\sigma y_1\leq_\sigma \dots \leq_\sigma y_n$, we have that
  $q< y_k$ for all $k\leq n$.
\end{proof}

Let $\mathscr{C}[X]=\big\{\mathscr{C}[x]: x\in X\big\}$ be the
decomposition space determined by the components of $X$. Recall that a
subset $\mathscr{U}\subset \mathscr{C}[X]$ is open in $\mathscr{C}[X]$
if $\bigcup\mathscr{U}$ is open in $X$. Alternatively,
$\mathscr{C}[X]$ is the quotient space obtained by the equivalence
relation $x\sim y$ iff $\mathscr{C}[x]=\mathscr{C}[y]$. Since the
elements of $\mathscr{C}[X]$ are closed sets, the decomposition space
$\mathscr{C}[X]$ is a $T_1$-space. The following property of the
decomposition space was essentially established in \cite[Corollary
3.7]{MR3640030}.

\begin{proposition}
  \label{proposition-KS-Coarser-v10:4}
  Let $X$ be a quasi-king space, and $\varphi$ be a continuous weak
  selection for $X$ such that $\mathscr{T}_\varphi$ is the topology of
  $X$. Then for every $x,y\in X$ with $\mathscr{C}[x]\cap
  \mathscr{C}[y]=\emptyset$ and $x<_\varphi y$, there are clopen
  sets $U,V\subset X$ such that $\mathscr{C}[x]\subset U$,
  $\mathscr{C}[y]\subset V$ and $U<_\varphi V$.
\end{proposition}

\begin{proof}
  Since $x<_\varphi y$, by Proposition
  \ref{proposition-KS-Coarser-v9:1}, we get that
  $\mathscr{C}[x]<_\varphi \mathscr{C}[y]$. Then the existence of such
  clopen sets $U,V\subset X$ follows by applying Lemma
  \ref{lemma-KS-Coarser-v16:1} and the condition that
  $\mathscr{T}_\varphi$ is the topology of $X$. Namely, by
  Proposition \ref{proposition-KS-Coarser-v5:2}, it suffices to
  construct open sets $U,V\subset X$ with $\mathscr{C}[x]\subset U$,
  $\mathscr{C}[y]\subset V$ and $U<_\varphi V$. Since $\mathscr{C}[y]$
  is compact (by Lemma \ref{lemma-KS-Coarser-v16:1}) and $\varphi$ is
  continuous, for each $z\in \mathscr{C}[x]$ there are open sets
  $U_z,V_z\subset X$ such that $z\in U_z$, $\mathscr{C}[y]\subset V_z$
  and $U_z<_\varphi V_z$. Finally, since $\mathscr{C}[x]$ is also
  compact, there exists a finite set $S\subset \mathscr{C}[x]$ with
  $\mathscr{C}[x]\subset \bigcup_{z\in S}U_z$. Then $U=\bigcup_{z\in
    S}U_z$ and $V=\bigcap_{z\in S}V_z$ are as required.
\end{proof}

The crucial final step in the preparation for the proof of Theorem
\ref{theorem-KS-Coarser-v10:1} is the following result.

\begin{lemma}
  \label{lemma-KS-Coarser-v10:1}
  Let $X$ be a selection-orderable quasi-king space.  Then the
 decomposition space $\mathscr{C}[X]$ is a zero-dimensional
 sequentially compact space.
\end{lemma}

\begin{proof}
  In this proof, we first show that $\mathscr{C}[X]$ has a continuous
  weak selection (following \cite[Theorem 3.1]{MR3640030}), and next
  that it is pseudocompact (following \cite[Theorem
  3.5]{MR2944781}). To this end, let $\varphi$ be a continuous weak
  selection for $X$ such that $\mathscr{T}_\varphi$ is the topology of
  $X$.  By Proposition \ref{proposition-KS-Coarser-v5:2}, each element
  of $\mathscr{C}[X]$ has a base of clopen sets. Hence, the
  decomposition space $\mathscr{C}[X]$ is zero-dimensional. Moreover,
  for every $x,y\in X$ with
  $\mathscr{C}[x]\cap \mathscr{C}[y]=\emptyset$ and $x<_\varphi y$,
  just as in the previous proof, we have that
  $\mathscr{C}[x]<_\varphi \mathscr{C}[y]$. Therefore, this defines a
  weak selection $\mathscr{C}[\varphi]$ for $\mathscr{C}[X]$ such that
  $\mathscr{C}[x]<_{\mathscr{C}[\varphi]} \mathscr{C}[y]$, whenever
  ${x<_\varphi y}$ with
  ${\mathscr{C}[x]\cap \mathscr{C}[y]=\emptyset}$. Finally, according
  to Proposition \ref{proposition-KS-Coarser-v10:4}, the selection
  $\mathscr{C}[\varphi]$ is continuous.\smallskip

  To show that $X$ is pseudocompact, take a discrete family
  $\{\mathscr{V}_n:n\in\N\}$ of nonempty open sets
  ${\mathscr{V}_n\subset \mathscr{C}[X]}$. Since $\mathscr{C}[X]$ is
  zero-dimensional, each $\mathscr{V}_n$, $n\in\N$, contains a
  nonempty clopen subset $\mathscr{U}_n\subset \mathscr{C}[X]$. Then
  each $U_n=\bigcup\mathscr{U}_n$, $n\in \N$, is a nonempty clopen
  subset of $X$, and the family ${\{U_n:n\in\N\}}$ is discrete in
  $X$. Hence, ${U_0=X\setminus\bigcup_{n\in\N}U_n}$ is also a clopen
  subset of $X$, and ${\{U_n:n<\omega\}}$ is an infinite pairwise
  disjoint open cover of $X$. According to Proposition
  \ref{proposition-KS-Coarser-v7:1}, this is impossible because $X$ is
  a quasi-king space. Thus, every discrete family of open subsets of
  $\mathscr{C}[X]$ is finite. Since $\mathscr{C}[X]$ is a Tychonoff
  space (being zero-dimensional), this implies that it is
  pseudocompact.\smallskip

  Having already established this, we can use each pseudocompact space
  with a continuous weak selection is sequentially compact
  \cite{artico-marconi-pelant-rotter-tkachenko:02,
    garcia-ferreira-sanchis:04, miyazaki:01b}, see also
  \cite[Corollary 3.9]{gutev-2013springer}. Accordingly,
  $\mathscr{C}[X]$ is sequentially compact.
\end{proof}

We now have also the proof of Theorem \ref{theorem-KS-Coarser-v10:1}.

\begin{proof}[Proof of Theorem \ref{theorem-KS-Coarser-v10:1}]
  Each pseudocompact space $X$ with a continuous weak selection is
  suborderable, see \cite{artico-marconi-pelant-rotter-tkachenko:02,
    garcia-ferreira-sanchis:04, miyazaki:01b}; also \cite[Theorems 3.7
  and 3.8]{gutev-2013springer}. Moreover, by Proposition
  \ref{proposition-KS-Coarser-v4:2}, each suborderable quasi-king
  space is compact. Hence, it suffices to show that $X$ is
  pseudocompact. In this, we follow the proof of \cite[Theorem
  4.1]{MR3640030}. Namely, assume to the contrary that $X$ is not
  pseudocompact. Then it has a continuous unbounded function
  $g:X\to [0,+\infty)$. Take a point $x_1\in X$ with $g(x_1)\geq 1$,
  and let $K_1=\mathscr{C}[x_1]$ be the component of $x_1$. Since $X$
  is a selection-orderable quasi-king space, by Lemma
  \ref{lemma-KS-Coarser-v16:1}, $K_1$ is compact, and consequently
  $g\uhr K_1$ is bounded. Hence, there exists a point
  $x_2\in X\setminus K_1$ with $g(x_2)\geq 2$. Set
  $K_2=\mathscr{C}[x_2]$ and extend the arguments by induction. Thus,
  there exists a pairwise disjoint sequence $\{K_n:n\in\N\}$ of
  components of $X$ and points $x_n\in K_n$ with $g(x_n)\geq n$, for
  every $n\in\N$. We claim that the sequence $\{K_n:n\in\N\}$ is
  discrete in $X$. Indeed, suppose that
  $y\in\overline{\bigcup_{n\geq k}K_n}\setminus\bigcup_{n\geq k}K_n$
  for some $k\in\N$. Since $\mathscr{C}[y]$ is compact,
  $g\uhr \mathscr{C}[y]$ is bounded, and so is $g\uhr U$ for some
  neighbourhood $U$ of $\mathscr{C}[y]$.  By Proposition
  \ref{proposition-KS-Coarser-v5:2}, this implies that $g\uhr H$ is
  bounded for some clopen subset $H\subset X$ with
  $\mathscr{C}[y]\subset H\subset U$.  However,
  $y\in H\cap \overline{\bigcup_{n\geq k}K_n}$ and, therefore, $H$
  meets infinitely many terms of the sequence $\{K_n:n\geq k\}$. In
  fact, being a clopen set, $H$ must contain infinitely many terms of
  this sequence because $ K_n\subset H$, whenever
  $H\cap K_n\neq \emptyset$.  Hence, $g\uhr H$ must be also unbounded
  because $g(x_n)\geq n$ for every $n\in\N$. A contradiction. Thus,
  $\{K_n:n\in\N\}$ is discrete.\smallskip

  We complete the proof as follows. Since
  $\{K_n:n\in \N\}\subset \mathscr{C}[X]$ is discrete in $X$, by
  Proposition \ref{proposition-KS-Coarser-v5:2}, it defines a discrete
  sequence of elements in the decomposition space
  $\mathscr{C}[X]$. However, this is impossible because, by Lemma
  \ref{lemma-KS-Coarser-v10:1}, the decomposition space
  $\mathscr{C}[X]$ is sequentially compact. We have duly arrived at a
  contradiction, showing that $X$ must be pseudocompact. 
\end{proof}

The proof of Theorem \ref{theorem-KS-Coarser-v11:1} now follows from
Theorem \ref{theorem-KS-Coarser-v3:1} and the following consequence of
Theorem \ref{theorem-KS-Coarser-v10:1}. 

\begin{corollary}
  \label{corollary-KS-Coarser-v11:1}
    Let $X$ be a quasi-king space with a properly continuous weak
  selection. Then  $X$ is weakly orderable. 
\end{corollary}

\begin{proof}
  Let $\varphi$ be a properly continuous weak selection for $X$. Then
  $\varphi$ is continuous with respect to its selection topology
  $\mathscr{T}_\varphi$. Moreover, $\mathscr{T}_\varphi$ is a coarser
  topology on $X$. Hence, $(X,\mathscr{T}_\varphi)$ remains a
  quasi-king space. Thus, $(X,\mathscr{T}_\varphi)$ is a
  selection-orderable quasi-king space and by Theorem
  \ref{theorem-KS-Coarser-v10:1}, it is compact. Finally, by a result
  of van Mill and Wattel \cite[Theorem 1.1]{mill-wattel:81},
  $(X,\mathscr{T}_\varphi)$ is an orderable space. Therefore, $X$ is
  weakly orderable.
\end{proof}


\newcommand{\noopsort}[1]{} \newcommand{\singleletter}[1]{#1}
\providecommand{\bysame}{\leavevmode\hbox to3em{\hrulefill}\thinspace}
\providecommand{\MR}{\relax\ifhmode\unskip\space\fi MR }
\providecommand{\MRhref}[2]{%
  \href{http://www.ams.org/mathscinet-getitem?mr=#1}{#2}
}
\providecommand{\href}[2]{#2}

\end{document}